\documentclass[12pt]{article}
\usepackage{amssymb}
\usepackage{amsthm}
\usepackage{amsmath}
\usepackage{graphicx}
\usepackage{enumerate}

\DeclareMathOperator{\Con}{Con}

\DeclareMathOperator{\Fil}{Fil}

\newtheorem{theorem}{Theorem}[section]
\newtheorem{definition}[theorem]{Definition}
\newtheorem{lemma}[theorem]{Lemma}
\newtheorem{proposition}[theorem]{Proposition}
\newtheorem{remark}[theorem]{Remark}
\newtheorem{example}[theorem]{Example}
\newtheorem{corollary}[theorem]{Corollary}
\title{Filters and congruences in sectionally pseudocomplemented lattices and posets}
\author{Ivan~Chajda and Helmut~L\"anger}
\date{}
\begin{document}
\footnotetext{Support of the research of the authors by the Austrian Science Fund (FWF), project I~4579-N, and the Czech Science Foundation (GA\v CR), project 20-09869L, entitled ``The many facets of orthomodularity'', as well as by \"OAD, project CZ~02/2019, entitled ``Function algebras and ordered structures related to logic and data fusion'', and, concerning the first author, by IGA, project P\v rF~2020~014, is gratefully acknowledged.}
\maketitle
\begin{abstract}
In our previous papers, together with J.~Paseka we introduced so-called sectionally pseudocomplemented lattices and posets and illuminated their role in algebraic constructions. We believe that -- similar to relatively pseudocomplemented lattices -- these structures can serve as an algebraic semantics of certain intuitionistic logics. The aim of the present paper is to define congruences and filters in these structures, derive mutual relationships between them and describe basic properties of congruences in strongly sectionally pseudocomplemented posets. For the description of filters both in sectionally pseudocomplemented lattices and posets, we use the tools introduced by A.~Ursini, i.e.\ ideal terms and the closedness with respect to them. It seems to be of some interest that a similar machinery can be applied also for strongly sectionally pseudocomplemented posets in spite of the fact that the corresponding ideal terms are not everywhere defined. 
\end{abstract}

{\bf AMS Subject Classification:} 06A11, 06D15, 06D20, 08B05, 08A30

{\bf Keywords:} Sectionally pseudocomplemented lattice, sectionally pseudocomplemented poset, filter, congruence, weak regularity, congruence permutability, Maltsev term, ideal term, closedness of a subset, congruence class, deductive system, partial term
\section{Introduction}

The concept of a relative pseudocomplemented lattice was introduced by R.~P.~Dilworth (\cite D). It was used in several branches of mathematics, e.g.\ as an algebraic axiomatization of intuitionistic logic (by Heyting and Brouwer) where the relative pseudocomplement is interpreted as the logical connective implication.

However, every relative pseudocomplemented lattice is distributive, see e.g.\ \cite B and \cite L. Because  not every non-classical propositional calculus is necessarily distributive (for instance, the logic of quantum mechanics), it was a question if the concept of relative pseudocomplementation can be extended in a reasonable way to non-distributive lattices. This was realized by the first author in \cite C by introducing sectional pseudocomplementation. Later on, the concept of sectional pseudocomplementation was extended also to posets, see \cite{CLP}.

In the present paper we focus on congruences and filters on sectionally pseudocomplemented lattices and posets. For lattices we can use the machinery of universal algebra (see e.g.\ \cite{CEL}) because sectionally pseudocomplemented lattices form a variety which is congruence permutable, congruence distributive and weakly regular. The situation with sectionally pseudocomplemented posets is a bit more complicated due to the fact that such a poset in general cannot be extended to a sectionally pseudocomplemented lattice by means of the Dedekind-MacNeille completion, see \cite{CLP} for the result.

\section{Sectionally pseudocomplemented lattices}

Recall that a {\em lattice} $(L,\vee,\wedge)$ is said to be {\em sectionally pseudocomplemented} if for all $a,b\in L$ there exists the pseudocomplement of $a\vee b$ in the interval $([b),\leq)$, i.e.\ the greatest element $c$ of $L$ satisfying
\[
(a\vee b)\wedge c=b.
\]
In this case $c$ is called the {\em sectional pseudocomplement of $a$ with respect to $b$} and it will be denoted by $a*b$. We consider sectionally pseudocomplemented lattices as algebras $(L,\vee,\wedge,*)$ of type $(2,2,2)$. Every non-empty sectionally pseudocomplemented lattice has a greatest element $1$, namely the algebraic constant $x*x$. In the following we consider only non-empty lattices.

An example of a sectionally pseudocomplemented lattice that is not relatively pseudocomplemented is $\mathbf N_5$ depicted in Figure~1:

\vspace*{-2mm}

\begin{center}
\setlength{\unitlength}{7mm}
\begin{picture}(4,8)
\put(2,1){\circle*{.3}}
\put(3,3){\circle*{.3}}
\put(1,4){\circle*{.3}}
\put(3,5){\circle*{.3}}
\put(2,7){\circle*{.3}}
\put(2,1){\line(-1,3)1}
\put(2,1){\line(1,2)1}
\put(2,7){\line(-1,-3)1}
\put(2,7){\line(1,-2)1}
\put(3,3){\line(0,1)2}
\put(1.85,.25){$0$}
\put(3.4,2.85){$a$}
\put(3.4,4.85){$c$}
\put(.3,3.85){$b$}
\put(1.85,7.4){$1$}
\put(1.2,-.75){{\rm Fig.\ 1}}
\end{picture}
\end{center}

\vspace*{4mm}

This lattice is not distributive and hence not relatively pseudocomplemented (see \cite B). The operation table for the sectional pseudocomplementation is as follows:
\[
\begin{array}{c|ccccc}
* & 0 & a & b & c & 1 \\
\hline
0 & 1 & 1 & 1 & 1 & 1 \\
a & b & 1 & b & 1 & 1 \\
b & c & a & 1 & c & 1 \\
c & b & a & b & 1 & 1 \\
1 & 0 & a & b & c & 1
\end{array}
\]
Recall from \cite{CLP}, Theorems~2.5 and 2.6, the following important result.

\begin{proposition}\label{prop1}
The class of sectionally pseudocomplemented lattices $(L,\vee,\wedge,*)$ forms a variety which besides the lattice axioms is determined by the following identities:
\begin{align*}
             z\vee y & \leq x*((x\vee y)\wedge(z\vee y)), \\
(x\vee y)\wedge(x*y) & \approx y.
\end{align*}
This variety is congruence permutable, congruence distributive and weakly regular. A Maltsev term for congruence permutability is given by
\[
p(x,y,z):=((x*y)*z)\wedge((z*y)*x).
\]
\end{proposition}

For the concept of congruence permutability we refer the reader to \cite{CEL}.

Weak regularity means that every congruence $\Theta$ on a sectionally pseudocomplemented lattice with greatest element $1$ is determined by its kernel, i.e.\ by the congruence class $[1]\Theta$. Hence our first task is to describe these classes. For this purpose we introduce the following concept:

\begin{definition}\label{def1}
Let $\mathbf L=(L,\vee,\wedge,*)$ be a sectionally pseudocomplemented lattice. A {\em filter} of $\mathbf L$ is a subset $F$ of $L$ containing $1$ such that $x*y,y*x\in F$ implies
\[
(x\vee z)*(y\vee z),(x\wedge z)*(y\wedge z),(x*z)*(y*z),(z*x)*(z*y)\in F.
\]
Let $\Fil\mathbf L$ denote the set of all filters of $\mathbf L$. For any subset $M$ of $L$ define a binary relation $\Phi(M)$ on $L$ as follows:
\[
\Phi(M):=\{(x,y)\in L^2\mid x*y,y*x\in M\}.
\]
\end{definition}

The following results were proved in \cite C and \cite{CLP}.

\begin{lemma}\label{lem1}
If $\mathbf L=(L,\vee,\wedge,*)$ is a sectionally pseudocomplemented lattice and $a,b,c\in L$ then
\begin{enumerate}[{\rm(i)}]
\item $a*b=1$ if and only if $a\leq b$,
\item $1*a=a$,
\item $a\leq b*a$,
\item $a\leq(a*b)*b$,
\item if $a\leq b$ then $b*c\leq a*c$,
\item $(a\vee b)\wedge(a*b)=b$.
\end{enumerate}
\end{lemma}

Observe that (iii) implies $b\leq(a*b)*b$.

The relationship between congruences and filters in sectionally pseudocomplemented lattices is illuminated in the next two theorems.

\begin{theorem}\label{th1}
Let $\mathbf L=(L,\vee,\wedge,*)$ be a sectionally pseudocomplemented lattice and $\Theta\in\Con\mathbf L$. Then $[1]\Theta\in\Fil\mathbf L$ and for any $x,y\in L$,
\[
(x,y)\in\Theta\text{ if and only if }x*y,y*x\in[1]\Theta,
\]
i.e.\ $\Phi([1]\Theta)=\Theta$.
\end{theorem}

\begin{proof}
Let $a,b\in L$. If $(a,b)\in\Theta$ then $a*b,b*a\in[a*a]\Theta=[1]\Theta$, i.e.\ $(a,b)\in\Phi([1]\Theta)$.
Conversely, if $(a,b)\in\Phi([1]\Theta)$ then $a*b,b*a\in[1]\Theta$ and hence, using (ii) and (iv) of Lemma~\ref{lem1},
\[
a=a\wedge((a*b)*b)\mathrel\Theta(1*a)\wedge(1*b)\mathrel\Theta((b*a)*a)\wedge b=b,
\]
i.e.\ $(a,b)\in\Theta$. This shows $\Phi([1]\Theta)=\Theta$. Due to the substitution property of $\Theta$ with respect to $\vee$, $\wedge$ and $*$ we see that $[1]\Theta$ satisfies the conditions from Definition~\ref{def1} and hence $[1]\Theta\in\Fil\mathbf L$.
\end{proof}

Theorem~\ref{th1} witnesses that sectionally pseudocomplemented lattices are weakly regular.

We can prove also the converse.

\begin{theorem}\label{th2}
Let $\mathbf L=(L,\vee,\wedge,*)$ be a sectionally pseudocomplemented lattice and $F\in\Fil\mathbf L$. Then $\Phi(F)\in\Con\mathbf L$ and $[1](\Phi(F))=F$.
\end{theorem}

\begin{proof}
Let $a,b,c\in L$. Evidently, $\Phi(F)$ is symmetric and since $1\in F$ and $x*x\approx1$ by (i) of Lemma~\ref{lem1}, it is also reflexive. Assume $a*b,b*a\in F$. Then by Definition~\ref{def1}
\begin{align*}
& (a*c)*(b*c),(b*c)*(a*c),(c*a)*(c*b),(c*b)*(c*a), \\
& (a\vee c)*(b\vee c),(b\vee c)*(a\vee c),(a\wedge c)*(b\wedge c),(b\wedge c)*(a\wedge c)\in F
\end{align*}
whence
\[
(a*c,b*c),(c*a,c*b),(a\vee c,b\vee c),(a\wedge c,b\wedge c)\in\Phi(F).
\]
Hence $\Phi(F)$ has the substitution property with respect to all basic operations of $\mathbf L$. Since the variety of sectionally pseudocomplemented lattices is congruence permutable, $\Phi(F)$ is also transitive, see e.g.\ Werner's Theorem (\cite W) or Corollary~3.1.13 in \cite{CEL}, and hence $\Phi(F)\in\Con\mathbf L$. Finally, the following are equivalent:
\begin{align*}
      a & \in[1](\Phi(F)), \\
  (a,1) & \in\Phi(F), \\
a*1,1*a & \in F, \\
    1,a & \in F, \\
		  a & \in F
\end{align*}
and hence $[1](\Phi(F))=F$.
\end{proof}

It is elementary to check that for every sectionally pseudocomplemented lattice $\mathbf L$, \\
$(\Fil\mathbf L,\subseteq)$ is a complete lattice.

\begin{example}
The sectionally pseudocomplemnted lattice from Fig.~1 has the following filters:
\begin{align*}
     F(1) & =\{1\}, \\
F(a)=F(c) & =\{a,c,1\}, \\
F(0)=F(b) & =\{0,a,b,c,1\}.
\end{align*}
\end{example}

The following corollary follows from Theorems~\ref{th1} and \ref{th2}.

\begin{corollary}\label{cor1}
For every sectionally pseudocomplemented lattice $\mathbf L$ the mappings $\Phi\mapsto[1]\Phi$ and $F\mapsto\Phi(F)$ are mutually inverse isomorphisms between the complete lattices $(\Con\mathbf L,\subseteq)$ and $(\Fil\mathbf L,\subseteq)$.
\end{corollary}

Let $(L,\vee,\wedge,*)$ be a sectionally pseudocomplemented lattice. A {\em deductive system} of $\mathbf L$ is a subset $D$ of $L$ containing $1$ and satisfying the following condition:
\[
\text{If }a\in D,b\in L\text{ and }a*b\in D\text{ then }b\in D.
\]
In the following $(F*(F*a))*a$ denotes the set $\{(x*(y*a))*a\mid x,y\in F\}$. Analogously, we proceed in similar cases.

\begin{theorem}
Let $\mathbf L=(L,\vee,\wedge,*)$ be a sectionally pseudocomplemented lattice, $\Theta\in\Con\mathbf L$, $F\in\Fil\mathbf L$ and $a,b\in L$. Then
\begin{enumerate}[{\rm(i)}]
\item Every class of $\Theta$ is a convex subset of $(L,\leq)$,
\item $F$ is a deductive system of $\mathbf L$,
\item $F$ is a lattice filter of $\mathbf L$,
\item $a*(F\wedge a)\subseteq F$ and $(F*(F*a))*a\subseteq F$.
\end{enumerate}
\end{theorem}

\begin{proof}
\
\begin{enumerate}[(i)]
\item If $c,d\in[a]\Theta$ and $c\leq b\leq d$ then
\[
b=c\vee b\in[d\vee b]\Theta=[d]\Theta=[a]\Theta.
\]
\item If $a,a*b\in F$ then
\[
b=1*b\in[a*b](\Phi(F))=[1](\Phi(F))=F.
\]
\item If $a\in F$ then
\[
a\vee b\in[1\vee b](\Phi(F))=[1](\Phi(F))=F.
\]
Moreover, if $a,b\in F$ then
\[
a\wedge b\in[1\wedge1](\Phi(F))=[1](\Phi(F))=F.
\]
\item
\begin{align*}
a*(F\wedge a) & \subseteq[a*(1\wedge a)](\Phi(F))=[a*a](\Phi(F))=[1](\Phi(F))=F, \\
  (F*(F*a))*a & \subseteq[(1*(1*a))*a](\Phi(F))=[(1*a)*a](\Phi(F))=[a*a](\Phi(F))= \\
	            & =[1](\Phi(F))=F.
\end{align*}
\end{enumerate}
\end{proof}

\section{Sectionally pseudocomplemented posets}

Now we turn our attention to sectionally pseudocomplemented posets.

\begin{definition}
Let $\mathbf P=(P,\leq)$ be a poset. Then $\mathbf P$ is called {\em sectionally pseudocomplemented} if for all $a,b\in P$ there exists a greatest element $c$ of $P$ satisfying
\[
L(U(a,b),c)=L(b).
\]
This element $c$ is called the {\em sectional pseudocomplement $a*b$ of $a$ with respect to $b$}. We write sectionally pseudocomplemented posets in the form $(P,\leq,*)$. A {\em strongly sectionally pseudocomplemented poset} is an ordered quadruple $(P,\leq,*,1)$ such that $(P,\leq,*)$ is a sectionally pseudocomplemented poset with greatest element $1$ satisfying the identity
\[
x\leq(x*y)*y.
\]
\end{definition}

The following results were proved in \cite{CLP}.

\begin{lemma}\label{lem4}
If $\mathbf P=(P,\leq,*)$ is a sectionally pseudocomplemented poset with greatest element $1$ and $a,b,c\in P$ then
\begin{enumerate}[{\rm(i)}]
\item $a*b=1$ if and only if $a\leq b$,
\item $1*a=a$,
\item $a\leq b*a$,
\item if $b\leq a$ then $a\leq(a*b)*b$,
\item if $a\leq b$ then $b*c\leq a*c$,
\item $L(U(a,b),a*b)=L(b)$.
\end{enumerate}
\end{lemma}

Observe that (iii) implies $b\leq(a*b)*b$. Hence in case $a\leq b$ we have $a\leq(a*b)*b$.

It is easy to see that every sectionally pseudocomplemented lattice is a strongly sectionally pseudocomplemented poset, and a lattice is sectionally pseudocomplemented if and only if it is sectionally pseudocomplemented as a poset.

\begin{remark}\label{rem1}
If $(P,\leq,*)$ is a sectionally pseudocomplemented poset and $a,b\in P$ then
\[
L(U(a,b),a*b)=L(b)
\]
which shows that there exists the infimum $U(a,b)\wedge(a*b)$ and hence the previous is equivalent to
\[
U(a,b)\wedge(a*b)=b.
\]
Thus, in case $a\geq b$ we obtain $a\wedge(a*b)=b$.
\end{remark}

An example of a strongly sectionally pseudocomplemented poset which is not a lattice is visualized in Figure~2.

\vspace*{-6mm}

\begin{center}
\setlength{\unitlength}{7mm}
\begin{picture}(4,10)
\put(2,1){\circle*{.3}}
\put(1,3){\circle*{.3}}
\put(3,3){\circle*{.3}}
\put(1,5){\circle*{.3}}
\put(1,7){\circle*{.3}}
\put(3,7){\circle*{.3}}
\put(2,9){\circle*{.3}}
\put(2,1){\line(-1,2)1}
\put(2,1){\line(1,2)1}
\put(1,3){\line(0,1)4}
\put(3,3){\line(-1,2)2}
\put(3,3){\line(0,1)4}
\put(1,5){\line(1,1)2}
\put(2,9){\line(-1,-2)1}
\put(2,9){\line(1,-2)1}
\put(1.85,.25){$0$}
\put(.3,2.85){$a$}
\put(3.4,2.85){$b$}
\put(.3,4.85){$c$}
\put(.3,6.85){$d$}
\put(3.4,6.85){$e$}
\put(1.85,9.4){$1$}
\put(1.2,-.75){{\rm Fig.\ 2}}
\end{picture}
\end{center}

\vspace*{4mm}

The operation table of $*$ is as follows:
\[
\begin{array}{c|ccccccc}
* & 0 & a & b & c & d & e & 1 \\
\hline
0 & 1 & 1 & 1 & 1 & 1 & 1 & 1 \\
a & b & 1 & b & 1 & 1 & 1 & 1 \\
b & c & a & 1 & c & 1 & 1 & 1 \\
c & b & a & b & 1 & 1 & 1 & 1 \\
d & 0 & a & b & c & 1 & e & 1 \\
e & 0 & a & b & c & d & 1 & 1 \\
1 & 0 & a & b & c & d & e & 1
\end{array}
\]
This poset is not relatively pseudocomplemented since the relative pseudocomplement of $c$ with respect to $a$ does not exist.

It should be noted that there are sectionally pseudocomplemented posets which are not strongly sectionally pseudocomplemented, see e.g.\ \cite{CLP}, but these are rather curious.

Since a sectionally pseudocomplemented poset $\mathbf P$ has only one operation, namely $*$, a congruence on $\mathbf P$ should satisfy the substitution property with respect to $*$. However, this condition is rather weak and we cannot expect to obtain a natural relationship between congruences and congruence kernels similar to that obtained for sectionally pseudocomplemented lattices in the previous section. Namely, our concept of a congruence on a strongly sectionally pseudocomplemented poset should respect also some aspects of the partial order relation. This is the reason why we introduce the following property.

\begin{definition}
A {\em binary relation} $\rho$ on a poset is called {\em $\min$-stable} if the following holds: If $(a,b),(c,d)\in\rho$, $a$ is comparable with $c$ and $b$ is comparable with $d$ then
\[
(\min(a,c),\min(b,d))\in\rho.
\]
\end{definition}

Observe that this condition trivially holds if $a\leq c$ and $b\leq d$ or if $a\geq c$ and $b\geq d$.

Now we can define

\begin{definition}
Let $\mathbf P=(P,\leq,*)$ be a sectionally pseudocomplemented poset. A {\em congruence} on $\mathbf P$ is a $\min$-stable congruence on the algebraic reduct $(P,*,1)$ of $\mathbf P$. Let $\Con\mathbf P$ denote the set of all congruences on $\mathbf P$.
\end{definition}

Note that the congruences on a sectionally pseudocomplemented lattice $\mathbf L$ may not coincide with the congruences on $\mathbf L$ if it is considered only as a sectionally pseudocomplemented poset.

In analogy to the lattice case we define

\begin{definition}\label{def2}
Let $\mathbf P=(P,\leq,*,1)$ be a sectionally pseudocomplemented poset with greatest element $1$. A {\em filter} of $\mathbf P$ is a subset $F$ of $P$ containing $1$ and satisfying the following conditions for all $x,y,z,v\in P$:
\begin{itemize}
\item If $x*y,y*x\in F$ then $(x*z)*(y*z),(z*x)*(z*y)\in F$,
\item if $x*y,y*x,z*v,v*z\in F$, $x$ and $z$ are comparable and $y$ and $v$ are comparable then $\min(x,z)*\min(y,v)\in F$.
\end{itemize}
Let $\Fil\mathbf P$ denote the set of all filters of $\mathbf P$. It is elementary to check that for every strongly sectionally pseudocomplemented poset $\mathbf P$, $(\Con\mathbf P,\subseteq)$ and $(\Fil\mathbf P,\subseteq)$ are complete lattices. For any subset $M$ of $P$ put
\[
\Phi(M):=\{(x,y)\in P^2\mid x*y,y*x\in M\}.
\]
\end{definition}

The relationship between congruences and filters in strongly sectionally pseudocomplemented posets is illuminated in the next two theorems.

\begin{theorem}\label{th3}
Let $\mathbf P=(P,\leq,*,1)$ be a strongly sectionally pseudocomplemented poset and $\Theta\in\Con\mathbf L$. Then $[1]\Theta\in\Fil\mathbf L$ and for any $x,y\in P$,
\[
(x,y)\in\Theta\text{ if and only if }x*y,y*x\in[1]\Theta,
\]
i.e.\ $\Phi([1]\Theta)=\Theta$.
\end{theorem}

\begin{proof}
Let $a,b\in L$. If $(a,b)\in\Theta$ then, by Lemma~\ref{lem4}, $a*b,b*a\in[a*a]\Theta=[1]\Theta$, i.e.\ $(a,b)\in\Phi([1]\Theta)$.
Conversely, if $(a,b)\in\Phi([1]\Theta)$ then $a*b,b*a\in[1]\Theta$ and hence, using again Lemma~\ref{lem4},
\begin{align*}
(a,(b*a)*a) & =(1*a,(b*a)*a)\in\Theta, \\
((a*b)*b,b) & =((a*b)*b,1*b)\in\Theta.
\end{align*}
Since $\mathbf P$ is strongly sectionally pseudocomplemented we have $a\leq(a*b)*b$ and $(b*a)*a\geq b$, thus by $\min$-stability of $\Theta$ we conclude
\[
(a,b)=(\min(a,(a*b)*b),\min((b*a)*a,b))\in\Theta.
\]
This shows $\Phi([1]\Theta)=\Theta$. Due to the substitution property of $\Theta$ with respect to $*$ and the $\min$-stability of $\Theta$ we obtain $[1]\Theta\in\Fil\mathbf L$.
\end{proof}

We have shown that every congruence $\Theta$ on a strongly sectionally pseudocomplemented poset is fully determined by its $1$-class $[1]\Theta$. Hence we conclude

\begin{corollary}
Strongly sectionally pseudocomplemented posets are weakly regular.
\end{corollary}

We can prove also the converse.

\begin{theorem}\label{th4}
Let $\mathbf P=(P,\leq,*,1)$ be a strongly sectionally pseudocomplemented poset and $F\in\Fil\mathbf P$. Then $\Phi(F)\in\Con\mathbf P$ and $[1](\Phi(F))=F$.
\end{theorem}

\begin{proof}
Let $a,b,c,d\in P$. Evidently, $\Phi(F)$ is symmetric and since $1\in F$ and $x*x\approx1$, it is also reflexive. If $(a,b)\in\Phi(F)$ then $a*b,b*a\in F$ and hence, using the properties listed in Definition~\ref{def2},
\begin{align*}
& (a*c)*(b*c),(b*c)*(a*c)\in F, \\
& (c*a)*(c*b),(c*b)*(c*a)\in F.
\end{align*}
Thus $(a*c,b*c),(c*a,c*b)\in\Phi(F)$. Hence $\Phi(F)$ has the substitution property with respect to $*$. Moreover, if $(a,b),(c,d)\in\Phi(F)$ then $a*b,b*a,c*d,d*c\in F$ and by Definition~\ref{def2}
\[
\min(a,c)*\min(b,d),\min(b,d)*\min(a,c)\in F,
\]
i.e.\ $(\min(a,c),\min(b,d))\in\Phi(F)$. This shows that $\Phi(F)$ is $\min$-stable. If $(a,b),(b,c)\in\Phi(F)$ then
\begin{align*}
      (a*b)*b & \mathrel{\Phi(F)}(b*b)*c=1*c=c, \\
a=1*a=(b*b)*a & \mathrel{\Phi(F)}(c*b)*b
\end{align*}
and hence using $\min$-stability of $\Phi(F)$
\[
(a,c)=(\min((a*b)*b,a),\min(c,(c*b)*b))\in\Phi(F),
\]
i.e.\ $\Phi(F)$ is transitive. Therefore $\Phi(F)\in\Con\mathbf P$. Finally, the following are equivalent:
\begin{align*}
      a & \in[1](\Phi(F)), \\
  (a,1) & \in\Phi(F), \\
a*1,1*a & \in F, \\
    1,a & \in F, \\
	    a & \in F.
\end{align*}
This shows $[1](\Phi(F))=F$.
\end{proof}

\begin{example}
The lattice of filters of the strongly sectionally pseudocomplemnted poset from Figure~2 consists of the following six filters:
\begin{align*}
      F(1) & =\{1\}, \\
      F(d) & =\{d,1\}, \\
      F(e) & =\{e,1\}, \\
F(\{d,e\}) & =\{d,e,1\}, \\
 F(a)=F(c) & =\{a,c,d,e,1\}, \\
 F(0)=F(b) & =\{0,a,b,c,d,e,1\}.
\end{align*}
The corresponding Hasse diagram is depicted in Figure~3:

\vspace*{-6mm}

\begin{center}
\setlength{\unitlength}{7mm}
\begin{picture}(4,10)
\put(2,1){\circle*{.3}}
\put(1,3){\circle*{.3}}
\put(3,3){\circle*{.3}}
\put(2,5){\circle*{.3}}
\put(2,7){\circle*{.3}}
\put(2,9){\circle*{.3}}
\put(2,1){\line(-1,2)1}
\put(2,1){\line(1,2)1}
\put(2,5){\line(-1,-2)1}
\put(2,5){\line(1,-2)1}
\put(2,5){\line(0,1)4}
\put(1.4,.25){$F(1)$}
\put(-.5,2.85){$F(d)$}
\put(3.3,2.85){$F(e)$}
\put(2.3,4.85){$F(\{d,e\})$}
\put(2.3,6.85){$F(a)$}
\put(1.4,9.4){$F(0)$}
\put(1.2,-.75){{\rm Fig.\ 3}}
\end{picture}
\end{center}

\vspace*{4mm}

\end{example}

The following corollary follows from Theorems~\ref{th3} and \ref{th4}.

\begin{corollary}
For every strongly sectionally pseudocomplemented poset $\mathbf P$ the mappings $\Phi\mapsto[1]\Phi$ and $F\mapsto\Phi(F)$ are mutually inverse isomorphisms between the complete lattices $(\Con\mathbf P,\subseteq)$ and $(\Fil\mathbf P,\subseteq)$.
\end{corollary}

\section{Properties of filters}

Using the $\min$-stability property of congruences in strongly sectionally pseudocomplemented posets we can prove

\begin{theorem}
Let $\mathbf P=(P,\leq,*,1)$ be a strongly sectionally pseudocomplemented poset and $\Theta\in\Con\mathbf P$. Then every class of $\Theta$ is a convex subset of $(P,\leq)$.
\end{theorem}

\begin{proof}
If $a,c\in P$, $b,d\in[a]\Theta$ and $b\leq c\leq d$ then
\begin{align*}
& (c*d)*b=1*b=b\leq c\leq(c*b)*b, \\
& ((c*d)*b,(c*b)*b)\in\Theta
\end{align*}
and hence by $\min$-stability of $\Theta$ we obtain
\[
(b,c)=(\min((c*d)*b,c),\min((c*b)*b,c))\in\Theta,
\]
which implies $c\in[b]\Theta=[a]\Theta$.
\end{proof}

We now investigate quotients $\mathbf P/\Theta$ of strongly sectionally pseudocomplemented posets $\mathbf P$ with respect to its congruences.

Let $\mathbf P=(P,\leq,*,1)$ be a strongly sectionally pseudocomplemented poset and $\Theta\in\Con\mathbf P$. We define a binary relation $\leq'$ on $P/\Theta$ by
\[
\text{for all }a,b\in P, [a]\Theta\leq'[b]\Theta\text{ if and only if }[a]\Theta*[b]\Theta=[1]\Theta.
\]
Recall that a {\em poset} $(P,\leq)$ is called up-directed if for any $x,y\in P$ there exists some $z\in P$ with $x,y\leq z$. Hence, every poset having a greatest element is up-directed.

It should be mentioned that the poset $(P/\Theta,\leq')$ where $\mathbf P=(P,\leq,*,1)$ denotes the strongly sectionally pseudocomplemented poset from Figure~2 and $\Theta$ the congruence on $\mathbf P$ corresponding to the filter $F(\{d,e\})$ of $\mathbf P$ is isomorphic to the lattice from Figure~1.

The following theorem was partly proved for congruences on the algebraic reduct $(P,*)$ in \cite{CLP}.

\begin{theorem}\label{th5}
Let $\mathbf P=(P,\leq,*,1)$ be a strongly sectionally pseudocomplemented poset, $n\geq1$, $a,a_1,\ldots,a_n,b\in P$ and $\Theta\in\Con\mathbf P$. Then the following hold:
\begin{enumerate}[{\rm(i)}]
\item if $a\leq b$ then $[a]\Theta\leq'[b]\Theta$,
\item $[a]\Theta\leq'[b]\Theta$ if and only if there exists some $c\in[b]\Theta$ with $a\leq c$,
\item $(P/\Theta,\leq')$ is a poset,
\item Every class of $\Theta$ is up-directed,
\item $U([a_1]\Theta,\ldots,[a_n]\Theta)=\{[x]\Theta\mid x\in U(a_1.\ldots,a_n)\}$ in $(P/\Theta,\leq')$.
\end{enumerate}
\end{theorem}

\begin{proof}
\
\begin{enumerate}[(i)]
\item If $a\leq b$ then $a*b=1$ whence $a*b\mathrel\Theta1$, i.e.\ $[a]\Theta*[b]\Theta=[a*b]\Theta=[1]\Theta$, thus $[a]\Theta\leq'[b]\Theta$.
\item If $[a]\Theta\leq'[b]\Theta$ then $a*b\mathrel\Theta1$ and hence $a\leq(a*b)*b\in[1*b]\Theta=[b]\Theta$. So one can put $c:=(a*b)*b$. If, conversely, there exists some $c\in[b]\Theta$ with $a\leq c$ then according to (i) we have $[a]\Theta\leq'[c]\Theta=[b]\Theta$.
\item Obviously, $\leq'$ is reflexive. Now assume $[a]\Theta\leq'[b]\Theta$ and $[b]\Theta\leq'[a]\Theta$. Then, by (ii), there exists some $c\in[b]\Theta$ with $a\leq c$. Because of $[c]\Theta=[b]\Theta\leq'[a]\Theta$ there exists some $d\in[a]\Theta$ with $c\leq d$. Since $a\leq c\leq d$, $a,d\in[a]\Theta$ and $([a]\Theta,\leq')$ is convex we conclude $c\in[a]\Theta$. Therefore $[a]\Theta=[c]\Theta=[b]\Theta$ which proves antisymmetry of $\leq'$. Finally, let $c\in P$ and assume $[a]\Theta\leq'[b]\Theta$ and $[b]\Theta\leq'[c]\Theta$. Then, by (ii) there exists some $e\in[b]\Theta$ with $a\leq e$ and because of $[e]\Theta=[b]\Theta\leq'[c]\Theta$ some $f\in[c]\Theta$ with $e\leq f$. From $a\leq e\leq f$ we have $a\leq f$ which implies $[a]\Theta\leq'[f]\Theta=[c]\Theta$ by (i), proving transitivity of $\leq'$.
\item Let $b,c\in[a]\Theta$. Then
\begin{align*}
& (b*c)*c\in[(c*c)*c]\Theta=[1*c]\Theta=[c]\Theta=[a]\Theta, \\
& b\leq(b*c)*c\text{ since }\mathbf P\text{ is strongly sectionally pseudocomplemented}, \\
& c\leq(b*c)*c\text{ according to Lemma~\ref{lem4} (iii)}.
\end{align*}
Thus $(b*c)*c$ is a common upper bound of $b$ and $c$ within $([a]\Theta,\leq)$.
\item Assume $[a]\Theta\in U([a_1]\Theta,\ldots,[a_n]\Theta)$. According to (ii), for all $i\in\{1,\ldots,n\}$ there exists some $b_i\in[a]\Theta$ with $a_i\leq b_i$. Because of (iv), $([a]\Theta,\leq')$ is up-directed and hence there exists some $c\in[a]\Theta$ with $b_1,\ldots,b_n\leq c$. This shows
\[
[a]\Theta=[c]\Theta\in\{[x]\Theta\mid x\in U(a_1,\ldots,a_n)\}.
\]
The converse inclusion follows from (i).
\end{enumerate}
\end{proof}

From (iv) we conclude that if $(P,\leq)$ satisfies the ascending chain condition (in particular, if $P$ is finite) then every class of $\Theta$ has a greatest element.

The following concept is inspired by the derivation rule Modus Ponens in the non-classical logic based on a sectionally pseudocomplemented poset where $*$ models the logical connective implication.

Let $(P,\leq,*,1)$ be a strongly sectionally pseudocomplemented poset. A {\em deductive system} of $\mathbf P$ is a subset $D$ of $P$ containing $1$ and satisfying the following condition:
\[
\text{If }a\in D,b\in P\text{ and }a*b\in D\text{ then }b\in D.
\]
We can prove the following result in analogy to the corresponding result for sectionally pseudocomplemented lattices.

\begin{theorem}\label{th7}
Let $\mathbf P=(P,\leq,*,1)$ be a strongly sectionally pseudocomplemented poset, $F\in\Fil\mathbf P$ and $c\in P$. Then
\begin{enumerate}[{\rm(i)}]
\item $F$ is a deductive system of $\mathbf P$,
\item $F$ is an order filter of $\mathbf P$,
\item $P*F\subseteq F$,
\item $c*(F\wedge c),(F*(F*c))*c\subseteq F$.
\end{enumerate}
\end{theorem}

\begin{proof}
We use the fact that the filter $F$ is the $1$-class of the congruence $\Phi(F)$.
\begin{enumerate}[(i)]
\item If $a\in F$, $b\in P$ and $a*b\in F$ then
\[
b=1*b\in[a*b](\Phi(F))=[1](\Phi(F))=F.
\]
\item If $a\in F$, $b\in P$ and $a\leq b$ then $a*b=1\in F$ and hence $b\in F$ by (i).
\item If $a\in P$ and $b\in F$ then $a*b\in[a*1](\Phi(F))=[1](\Phi(F))=F$.
\item
\begin{align*}
c*(F\wedge c) & \subseteq[c*(1\wedge c)](\Phi(F))=[c*c](\Phi(F))=[1](\Phi(F))=F, \\
  (F*(F*c))*c & \subseteq[(1*(1*c))*c](\Phi(F))=[(1*c)*c](\Phi(F))=[c*c](\Phi(F))= \\
	            & =[1](\Phi(F))=F.
\end{align*}
\end{enumerate}
\end{proof}

Theorem~\ref{th7} shows that every filter is a deductive system. However, our concept of a filter is rather complicated and it seems that not all the properties of a filter are necessary to prove this assertion. We can prove

\begin{proposition}\label{prop2}
Let $\mathbf P=(P,\leq,*,1)$ be a strongly sectionally pseudocomplemented poset and $M$ a subset of $P$ containing $1$ and satisfying $(M*(M*x))*x\subseteq M$ for all $x\in P$. Then $M$ is a deductive system of $\mathbf P$.
\end{proposition}

\begin{proof}
Let $a\in N$ and $b\in P$. We have $1\in M$. If $a\leq b$ then
\[
b=1*b=(a*b)*b=(a*(1*b))*b\in(N*(N*b))*b\subseteq N.
\]
Hence, if $a*b\in N$ then because of $a\leq(a*b)*b$ we have $(a*b)*b\in N$ which implies
\[
b=1*b=(((a*b)*b)((a*b)*b))*b\in(N*(N*b))*b\subseteq N.
\]
\end{proof}

Observe that the condition mentioned in Proposition~\ref{prop2} is just the second one of (iv) of Theorem~\ref{th7}.

For the concept of an ideal of a universal algebra which corresponds to our concept of a filter and for the concept of ideal terms the reader is referred to \cite U. In particular, for ideals (alias filters) in permutable and weakly regular varieties see also \cite{CEL} for details.
 
\begin{definition}
An {\em ideal term} for sectionally pseudocomplemented lattices is a term $t(x_1,\ldots,x_n,y_1,\ldots,y_m)$ in the language of sectionally pseudocomplemented lattices satisfying the identity
\[
t(x_1,\ldots,x_n,1,\ldots,1)\approx1.
\]
\end{definition}

Of course, there exists an infinite number of ideal terms in sectionally pseudocomplemented lattices. The following list including five ideal terms is a so-called {\em basis for filters} in sectionally pseudocomplemented lattices, i.e.\ filters can be characterized by this short list of ideal terms.

\begin{lemma}\label{lem2}
The following terms are ideal terms for sectionally pseudocomplemented lattices:
\begin{align*}
                     t_1 & :=1, \\
t_2(x_1,x_2,x_3,y_1,y_2) & :=(((x_1\vee x_2)\wedge(y_1*x_2)\wedge y_2)\vee x_3)*(x_2\vee x_3), \\
t_3(x_1,x_2,x_3,y_1,y_2) & :=(((x_1\vee x_2)\wedge(y_1*x_2)\wedge y_2)\wedge x_3)*(x_2\wedge x_3), \\
t_4(x_1,x_2,x_3,y_1,y_2) & :=(((x_1\vee x_2)\wedge(y_1*x_2)\wedge y_2)*x_3)*(x_2*x_3), \\
t_5(x_1,x_2,x_3,y_1,y_2) & :=(x_3*x_1)*(x_3*(((x_1\vee x_2)\wedge(y_2*x_1)\wedge y_1)).
\end{align*}
\end{lemma}

\begin{proof}
Put
\[
t(x,y,z,u):=(x\vee y)\wedge(z*y)\wedge u.
\]
Then
\begin{align*}
t_2(x_1,x_2,x_3,y_1,y_2) & =(t(x_1,x_2,y_1,y_2)\vee x_3)*(x_2\vee x_3), \\
t_3(x_1,x_2,x_3,y_1,y_2) & =(t(x_1,x_2,y_1,y_2)\wedge x_3)*(x_2\wedge x_3), \\
t_4(x_1,x_2,x_3,y_1,y_2) & =(t(x_1,x_2,y_1,y_2)*x_3)*(x_2*x_3), \\
t_5(x_1,x_2,x_3,y_1,y_2) & =(x_3*x_1)*(x_3*t(x_2,x_1,y_2,y_1)).
\end{align*}
and according to Lemma~\ref{lem1}
\[
t(x,y,1,1)=(x\vee y)\wedge(1*y)\wedge1=(x\vee y)\wedge y=y
\]
and hence
\begin{align*}
t_2(x_1,x_2,x_3,1,1) & =(t(x_1,x_2,1,1)\vee x_3)*(x_2\vee x_3)=(x_2\vee x_3)*(x_2\vee x_3)=1, \\
t_3(x_1,x_2,x_3,1,1) & =(t(x_1,x_2,1,1)\wedge x_3)*(x_2\wedge x_3)=(x_2\wedge x_3)*(x_2\wedge x_3)=1, \\
t_4(x_1,x_2,x_3,1,1) & =(t(x_1,x_2,1,1)*x_3)*(x_2*x_3)=(x_2*x_3)*(x_2*x_3)=1, \\
t_5(x_1,x_2,x_3,1,1) & =(x_3*x_1)*(x_3*t(x_2,x_1,1,1))=(x_3*x_1)*(x_3*x_1)=1.
\end{align*}
\end{proof}

The closedness with respect to ideal terms was also introduced by A.~Ursini (\cite U).

\begin{definition}
A {\em subset} $A$ of a sectionally pseudocomplemented lattice $\mathbf L=(L,\vee,\wedge,*)$ is said to be {\em closed} with respect to the ideal terms $t_i(x_1,\ldots,x_n,y_1,\ldots,y_m)$, $i\in I$, if for every $i\in I$, all $x_1,\ldots,x_n\in L$ and all $y_1,\ldots,y_m\in A$ we have $t_i(x_1,\ldots,x_n,y_1,\ldots,y_m)\in A$.
\end{definition}

Now we prove that the ideal terms listed in Lemma~\ref{lem2} form a basis for filters, i.e.\ filters are characterized as those subsets which are closed with respect to these ideal terms.

\begin{theorem}\label{th6}
Let $\mathbf L=(L,\vee,\wedge,*)$ be a sectionally pseudocomplemented lattice and $F\subseteq L$. Then $F\in\Fil\mathbf L$ if and only if $F$ is closed with respect to the ideal terms $t_1,\ldots,t_5$ listed in Lemma~\ref{lem2}.
\end{theorem}

\begin{proof}
If $F\in\Fil\mathbf L$ then $F=[1](\Phi(F))$ according to Theorem~\ref{th2}, and if
\[
t_i(x_1,\ldots,x_n,y_1,\ldots,y_m),\quad i\in\{1,\ldots,5\},
\]
are the ideal terms listed in Lemma~\ref{lem2}, $a_1,\ldots,a_n\in L$ and $b_1,\ldots,b_m\in F$ then
\[
t_i(a_1,\ldots,a_n,b_1,\ldots,b_m)\in[t_i(a_1,\ldots,a_n,1,\ldots,1)](\Phi(F))=[1](\Phi(F))=F
\]
according to Lemma~\ref{lem2} and hence $F$ is closed with respect to the ideal terms $t_1,\ldots,t_5$. Conversely, assume $F$ to be closed with respect to the ideal terms $t_1,\ldots,t_5$. Then $1=t_1\in F$. Now assume $a,b\in L$ and $a*b,b*a\in F$. For the term
\[
t(x,y,z,u):=(x\vee y)\wedge(z*y)\wedge u
\]
we have
\begin{align*}
t_2(x_1,x_2,x_3,y_1,y_2) & =(t(x_1,x_2,y_1,y_2)\vee x_3)*(x_2\vee x_3), \\
t_3(x_1,x_2,x_3,y_1,y_2) & =(t(x_1,x_2,y_1,y_2)\wedge x_3)*(x_2\wedge x_3), \\
t_4(x_1,x_2,x_3,y_1,y_2) & =(t(x_1,x_2,y_1,y_2)*x_3)*(x_2*x_3), \\
t_5(x_1,x_2,x_3,y_1,y_2) & =(x_3*x_1)*(x_3*t(x_2,x_1,y_2,y_1))
\end{align*}
and according to Lemma~\ref{lem1} (iv) and (vi) we obtain
\begin{align*}
t(x,y,x*y,y*x) & =(x\vee y)\wedge((x*y)*y)\wedge(y*x)= \\
               & =((y\vee x)\wedge(y*x))\wedge((x*y)*y)=x\wedge((x*y)*y)=x.
\end{align*}
Hence
\begin{align*}
    (a\vee c)*(b\vee c) & =(t(a,b,a*b,b*a)\vee c)*(b\vee c)=t_2(a,b,c,a*b,b*a)\in F, \\
(a\wedge c)*(b\wedge c) & =(t(a,b,a*b,b*a)\wedge c)*(b\wedge c)=t_  3(a,b,c,a*b,b*a)\in F, \\
            (a*c)*(b*c) & =(t(a,b,a*b,b*a)*c)*(b*c)=t_4(a,b,c,a*b,b*a)\in F, \\
            (c*a)*(c*b) & =(c*a)*(c*t(b,a,b*a,a*b))=t_5(a,b,c,a*b,b*a)\in F
\end{align*}
showing $F\in\Fil\mathbf L$.
\end{proof}

\begin{remark}
Let us note that the term $t$ from the proof of Theorem~\ref{th6} gives rise to a Maltsev term. Namely, if
\begin{align*}
t(x,y,z,u) & :=(x\vee y)\wedge(z*y)\wedge u\text{ and} \\
  q(x,y,z) & :=t(x,z,x*y,y*x).
\end{align*}
then
\begin{align*}
q(x,y,z) & =(x\vee z)\wedge((x*y)*z)\wedge(y*x), \\
q(x,x,z) & =(x\vee z)\wedge((x*x)*z)\wedge(x*x)=(x\vee z)\wedge(1*z)\wedge1=(x\vee z)\wedge z=z, \\
q(x,z,z) & =(x\vee z)\wedge((x*z)*z)\wedge(z*x)=((z\vee x)\wedge(z*x))\wedge((x*z)*z)= \\
         & =x\wedge((x*z)*z)=x.
\end{align*}
Observe that the Maltsev term $q(x,y,z)$ is different from that in Proposition~\ref{prop1}.
\end{remark}

In the following we write $a\wedge b\wedge c$ instead of $\inf(a,b,c)$.

Now we introduce a certain modification of the notion an ideal term (for posets) which need not be defined everywhere. This will be used in the sequel.

\begin{definition}
A {\em partial ideal term} for sectionally pseudocomplemented posets with greatest element $1$ is a partially defined term $T(x_1,\ldots,x_n,y_1,\ldots,y_m)$ in the language of sectionally pseudocomplemented posets with greatest element $1$ satisfying the identity
\[
T(x_1,\ldots,x_n,1,\ldots,1)\approx1.
\]
This language contains also a binary operator $U(x,y)$.
\end{definition}

Using of the concept of partial ideal terms, we will try to describe filters also in strongly sectionally pseudocomplemented posets. Similarly as in Lemma~\ref{lem2} we firstly get a list of four partial ideal terms which will be shown to suffice.

\begin{lemma}\label{lem3}
The following partial terms are partial ideal terms for strongly sectionally pseudocomplemented posets:
\begin{align*}
                                 T_1 & :=1, \\
            T_2(x_1,x_2,x_3,y_1,y_2) & :=((U(x_1,x_2)\wedge(y_1*x_2)\wedge y_2)*x_3)*(x_2*x_3), \\
            T_3(x_1,x_2,x_3,y_1,y_2) & :=(x_3*x_1)*(x_3*(U(x_1,x_2)\wedge(y_2*x_1)\wedge y_1)), \\
T_4(x_1,x_2,x_3,x_4,y_1,y_2,y_3,y_4) & :=((U(x_1,x_2)\wedge(y_1*x_2)\wedge y_2)\wedge \\
                                     & \hspace*{6.75mm} \wedge(U(x_3,x_4)\wedge(y_3*x_4)\wedge y_4))*(x_2\wedge x_4).
\end{align*}
\end{lemma}

\begin{proof}
Put
\[
T(x,y,z,u):=U(x,y)\wedge(z*y)\wedge u.
\]
Then
\begin{align*}
            T_2(x_1,x_2,x_3,y_1,y_2) & =(T(x_1,x_2,y_1,y_2)*x_3)*(x_2*x_3), \\
            T_3(x_1,x_2,x_3,y_1,y_2) & =(x_3*x_1)*(x_3*T(x_2,x_1,y_2,y_1)), \\
T_4(x_1,x_2,x_3,x_4,y_1,y_2,y_3,y_4) & =(T(x_1,x_2,y_1,y_2)\wedge T(x_3,x_4,y_3,y_4))*(x_2\wedge x_4)
\end{align*}
and according to Lemma~\ref{lem4} and Remark~\ref{rem1}
\[
T(x,y,1,1)=U(x,y)\wedge(1*y)\wedge1=U(x,y)\wedge y=y.
\]
Hence
\begin{align*}
        T_2(x_1,x_2,x_3,1,1) & =(T(x_1,x_2,1,1)*x_3)*(x_2*x_3)=(x_2*x_3)*(x_2*x_3)=1, \\
        T_3(x_1,x_2,x_3,1,1) & =(x_3*x_1)*(x_3*T(x_2,x_1,1,1))=(x_3*x_1)*(x_3*x_1)=1, \\
T_4(x_1,x_2,x_3,x_4,1,1,1,1) & =(T(x_1,x_2,1,1)\wedge T(x_3,x_4,1,1))*(x_2\wedge x_4)= \\
                             & =(x_2*x_4)*(x_2*x_4)=1.
\end{align*}
\end{proof}

Now we define closedness with respect to partial ideal terms.

\begin{definition}
A {\em subset} $A$ of a strongly sectionally pseudocomplemented poset $\mathbf P=(P,\leq,*,1)$ is said to be {\em closed} with respect to the partial ideal terms $T_i(x_1,\ldots,x_n,y_1,\ldots$ $\ldots,y_m)$, $i\in I$, if for every $i\in I$, all $x_1,\ldots,x_n\in P$ and all $y_1,\ldots,y_m\in A$ we have that $T_i(x_1,\ldots,x_n,y_1,\ldots,y_m)$ is defined and $T_i(x_1,\ldots,x_n,y_1,\ldots,y_m)\in A$.
\end{definition}

Although our ideal terms are only partial, we can prove that every subset of a strongly sectionally pseudocomplemented poset $\mathbf P$ closed with respect to them is really a filter of $\mathbf P$.

\begin{theorem}
Let $\mathbf P=(P,\leq,*,1)$ be a strongly sectionally pseudocomplemented poset and $F$ a subset of $P$ that is closed with respect to the partial ideal terms $T_1,\ldots,T_4$ listed in Lemma~\ref{lem3}. Then $F\in\Fil\mathbf P$.
\end{theorem}

\begin{proof}
We have $1=T_1\in F$. Now assume $a,b,c,d\in P$ and $a*b,b*a,c*d,d*c\in F$. For the partial term
\[
T(x,y,z,u):=U(x,y)\wedge(z*y)\wedge u
\]
we have
\begin{align*}
            T_2(x_1,x_2,x_3,y_1,y_2) & =(T(x_1,x_2,y_1,y_2)*x_3)*(x_2*x_3), \\
            T_3(x_1,x_2,x_3,y_1,y_2) & =(x_3*x_1)*(x_3*T(x_2,x_1,y_2,y_1)), \\
T_4(x_1,x_2,x_3,x_4,y_1,y_2,y_3,y_4) & =(T(x_1,x_2,y_1,y_2)\wedge T(x_3,x_4,y_3,y_4))*(x_2\wedge x_4)
\end{align*}
and according to Lemma~\ref{lem4} and Remark~\ref{rem1} we obtain
\begin{align*}
T(x,y,x*y,y*x) & =U(x,y)\wedge((x*y)*y)\wedge(y*x)= \\
               & =(U(y,x)\wedge(y*x))\wedge((x*y)*y)=x\wedge((x*y)*y)=x.
\end{align*}
Hence
\begin{align*}
            (a*c)*(b*c) & =(T(a,b,a*b,b*a)*c)*(b*c)=T_2(a,b,c,a*b,b*a)\in F, \\
            (c*a)*(c*b) & =(c*a)*(c*T(b,a,b*a,a*b))=T_3(a,b,c,a*b,b*a)\in F.
\end{align*}
Moreover, if $a$ and $c$ are comparable and $b$ and $d$ are comparable then we apply the partial term $T_4$ to derive
\begin{align*}
\min(a,c)*\min(b,d) & =(T(a,b,a*b,b*a)\wedge T(c,d,c*d,d*c))*(b\wedge d)= \\
                    & =T_4(a,b,c,d,a*b,b*a,c*d,d*c)\in F.
\end{align*}
This shows $F\in\Fil\mathbf P$.
\end{proof}

\begin{remark}
Let us consider the partial term $T(x,y,z,u):=U(x,y)\wedge(z*y)\wedge u$ from the proof of Lemma~\ref{lem3} and put
\[
Q(x,y,z):=T(x,z,x*y,y*x),
\]
i.e.
\[
Q(x,y,z)=U(x,z)\wedge((x*y)*z)\wedge(y*x).
\]
Of course, this is only a partial term because the infimum in $Q$ need not exists for some elements from a strongly sectionally pseudocomplemented poset $\mathbf P=(P,\leq,*,1)$. It is of some interest that this partial term behaves like a Maltsev term. Namely, we can easily compute
\begin{align*}
Q(x,x,z) & =U(x,z)\wedge((x*x)*z)\wedge(x*x)=U(x,z)\wedge(1*z)\wedge1=U(x,y)\wedge z=z, \\
Q(x,z,z) & =U(x,z)\wedge((x*z)*z)\wedge(z*x)=(U(z,x)\wedge(z*x))\wedge((x*z)*z)= \\
         & =x\wedge((x*z)*z)=x.
\end{align*}
Moreover, these expressions $Q(x,x,z)$ and $Q(x,z,z)$ are defined for all $x,z\in P$.
\end{remark}

For every sectionally pseudocomplemented lattice $\mathbf L=(L,\vee,\wedge,*)$ and every $M\subseteq L$ let $F(M)$ denote the filter of $\mathbf L$ generated by $M$.

The connection between filters generated by a certain subset and congruences on sectionally pseudocomplemenetd lattices is described in the following proposition.

\begin{proposition}
Let $\mathbf L=(L,\vee,\wedge,*)$ be a sectionally pseudocomplemented lattice, $M\subseteq L$ and $a\in L$. Then
\begin{align*}
               \Phi(F(M)) & =\Theta(M\times\{1\}), \\
[1](\Theta(M\times\{1\})) & =F(M).
\end{align*}
In particular,
\begin{align*}
      \Phi(F(a)) & =\Theta(a,1), \\
[1](\Theta(a,1)) & =F(a).
\end{align*}
\end{proposition}

\begin{proof}
Since $M\times\{1\}\subseteq\Phi(F(M))$ we have
\[
\Theta(M\times\{1\})\subseteq\Phi(F(M))
\]
and hence
\[
[1](\Theta(M\times\{1\}))\subseteq[1](\Phi(F(M)))=F(M)
\]
according to Corollary~\ref{cor1}. Because of $M\subseteq[1](\Theta(M\times\{1\}))$ we have
\[
F(M)\subseteq[1](\Theta(M\times\{1\}))
\]
and hence
\[
\Phi(F(M))\subseteq\Phi([1](\Theta(M\times\{1\})))=\Theta(M\times\{1\})
\]
according to Corollary~\ref{cor1}.
\end{proof}

An analogous result holds for strongly sectionally pseudocomplemented posets.

Authors' addresses:

Ivan Chajda \\
Palack\'y University Olomouc \\
Faculty of Science \\
Department of Algebra and Geometry \\
17.\ listopadu 12 \\
771 46 Olomouc \\
Czech Republic \\
ivan.chajda@upol.cz

Helmut L\"anger \\
TU Wien \\
Faculty of Mathematics and Geoinformation \\
Institute of Discrete Mathematics and Geometry \\
Wiedner Hauptstra\ss e 8-10 \\
1040 Vienna \\
Austria, and \\
Palack\'y University Olomouc \\
Faculty of Science \\
Department of Algebra and Geometry \\
17.\ listopadu 12 \\
771 46 Olomouc \\
Czech Republic \\
helmut.laenger@tuwien.ac.at
\end{document}